\documentclass[12pt,a4paper]{article}
\usepackage[top=2.5cm,bottom=2.5cm,left=2.5cm,right=2.5cm]{geometry}

\usepackage{amssymb,amsmath,graphicx,color,amsthm,centernot}
\usepackage[capitalise]{cleveref}
\usepackage[labelformat=simple]{subcaption}
\usepackage[font=small,labelfont=bf,width=12.5cm]{caption}
\usepackage{multirow,tabularx}

\newtheorem{claim}{Claim}
\newtheorem{theorem}{Theorem}[section]
\newtheorem{lemma}[theorem]{Lemma}

\newtheorem{proposition}[theorem]{Proposition}

\newtheorem{conjecture}[theorem]{Conjecture}
\newtheorem{question}[theorem]{Question}

\newcommand{\set}[1]{\ensuremath{\left\{#1 \right\}}}

\newenvironment{proofclaim}[1][]%
    {\noindent \emph{Proof.} {}{#1}{}}{$~$\hfill $~\blacklozenge$ \vspace{0.2cm}}

\definecolor{defblue}{rgb}{0.4,0,0.84}
\definecolor{greyblue}{rgb}{0.23,0.4,0.70}
\definecolor{orange}{rgb}{1.0,0.5,0.2}
\definecolor{violet}{rgb}{0.55,0,0.55}

\newcommand{\lir}{\ensuremath{\chi_{\mathrm{irr}}'}}

\usepackage{url}
\makeatletter
\g@addto@macro{\UrlBreaks}{\UrlOrds}
\makeatother

\newcolumntype{Y}{>{\centering\arraybackslash}X}

\begin{document}

\title{{\bf Locally irregular edge-coloring \\ of subcubic graphs}}

\author
{	
	Borut Lu\v{z}ar\thanks{Faculty of Information Studies in Novo mesto, Slovenia. 
		E-Mails: \texttt{borut.luzar@gmail.com, kennystorgel.research@gmail.com}}, \quad
	M\'{a}ria Macekov\'{a}\thanks{Faculty of Science, Pavol Jozef \v Saf\'{a}rik University, Ko\v{s}ice, Slovakia, \newline 
		E-Mails: \texttt{\{maria.macekova,roman.sotak\}@upjs.sk}, \texttt{simona.rindosova@student.upjs.sk}, \texttt{k.cekanova@gmail.com}}, \quad
	Simona Rindo\v sov\'{a}\footnotemark[2], \\
	Roman Sot\'{a}k\footnotemark[2], \quad
	Katar\'{i}na Srokov\'{a}\footnotemark[2], \quad
	Kenny \v{S}torgel\footnotemark[1]
}

\maketitle

{
\begin{abstract}
	A graph is {\em locally irregular} if no two adjacent vertices have the same degree.
	A {\em locally irregular edge-coloring} of a graph $G$ is such an (improper) edge-coloring 
	that the edges of any fixed color induce a locally irregular graph. 
	Among the graphs admitting a locally irregular edge-coloring, i.e., {\em decomposable graphs},
	only one is known to require $4$ colors, while for all the others it is believed that $3$ colors suffice.
	In this paper, we prove that decomposable claw-free graphs with maximum degree $3$,
	all cycle permutation graphs, and all generalized Petersen graphs admit a locally irregular edge-coloring with at most $3$ colors.
	We also discuss when $2$ colors suffice for a locally irregular edge-coloring of cubic graphs 
	and present an infinite family of cubic graphs of girth $4$ which require $3$ colors.
\end{abstract}
}

\medskip
{\noindent\small \textbf{Keywords:} locally irregular edge-coloring, locally irregular graph, subcubic graph, generalized Petersen graph, cycle permutation graph}

\section{Introduction}

A graph is {\em locally irregular} if no two adjacent vertices have the same degree.
A {\em locally irregular edge-coloring} (or a {\em LIEC} for short) of a graph $G$ is such an edge-coloring 
that the edges of any fixed color induce a locally irregular graph. 
Since $K_2$ is not a locally irregular graph, every LIEC is necessarily improper.
Hence, there are graphs (e.g., odd paths and odd cycles) that do not admit a LIEC.
If a graph does admit a LIEC, we call it {\em decomposable}.
The smallest number of colors such that a decomposable graph $G$ admits 
a LIEC is called the \textit{locally irregular chromatic index} and denoted by $\lir(G)$.

The locally irregular edge-coloring was introduced in 2015 by Baudon et al.~\cite{BauBenPrzWoz15} 
who were motivated by the well-known (1-2-3)-Conjecture proposed in~\cite{KarLucTho04}.
Since then, a number of results were published.
The first to establish a constant upper bound of $328$ for the locally irregular chromatic index of decomposable graphs
were Bensmail, Merker, and Thomassen~\cite{BenMerTho16}. 
Their bound was further improved to $220$ by Lu\v{z}ar, Przyby{\l}o, and Sot\'{a}k~\cite{LuzPrzSot18},
who combined the decomposition presented in~\cite{BenMerTho16} 
with a result on decomposition of bipartite graphs to odd subgraphs presented in~\cite{LuzPetSkr15}.

It turned out that only a few colors usually suffice for a LIEC of a decomposable graph,
and so Baudon et al.~\cite{BauBenPrzWoz15} conjectured that 
every decomposable graph admits a LIEC using at most $3$ colors. 
Their conjecture was just recently rejected by Sedlar and \v{S}krekovski~\cite{SedSkr21} 
who presented the graph $H_0$ with $\lir(H_0) = 4$ (see Figure~\ref{fig:counter}).
\begin{figure}[htp!]
	$$
		\includegraphics{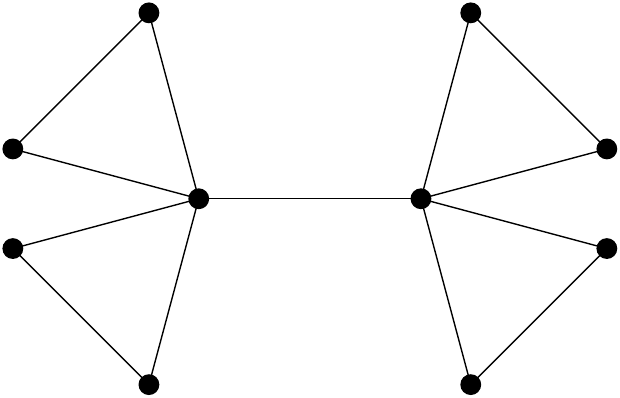}
	$$
	\caption{The graph $H_0$ with $\lir(H_0) = 4$.}
	\label{fig:counter}
\end{figure}

The graph $H_0$ is a cactus graph and after establishing that all decomposable cacti except $H_0$ admit a LIEC 
with at most $3$ colors~\cite{SedSkr22b} (see also~\cite{SedSkr22a}), 
Sedlar and \v{S}krekovski proposed a revised conjecture.
\begin{conjecture}[Sedlar and \v{S}krekovski~\cite{SedSkr22b}]
	\label{conj:main2}
	For every connected decomposable graph $G$, not isomorphic to $H_0$, it holds that $\lir(G) \le 3$.
\end{conjecture}

Regarding decomposable graphs admitting locally irregular edge-colorings using at most $3$ colors,
it was shown that, e.g., $k$-regular graphs for $k \ge 10^7$~\cite{BauBenPrzWoz15},
graphs with minimum degree at least $10^{10}$~\cite{Prz16},
trees~\cite{BauBenSop15}, and unicyclic graphs~\cite{SedSkr21} are such.

In this paper, we focus on graphs with maximum degree $3$.
In~\cite{LuzPrzSot18}, the authors proved the following.
\begin{theorem}[Lu\v{z}ar, Przyby{\l}o, and Sot\'{a}k~\cite{LuzPrzSot18}]
	\label{thm:subcub}
	For every decomposable graph $G$ with maximum degree $3$ it holds that $\lir(G) \le 4$.
\end{theorem}
It remained an open question whether $3$ colors are always sufficient. 
We partially answer this by proving that the answer is affirmative for
decomposable claw-free graphs with maximum degree $3$ in Section~\ref{sec:claw}, 
and for cycle permutation graphs and generalized Petersen graphs in Section~\ref{sec:Petersen}.

We additionally study cubic graphs that do not admit a LIEC using at most $2$ colors.
It is already known that every cubic bipartite graph $G$ has $\lir(G) \le 2$~\cite{BauBenPrzWoz15}.
This motivated us to investigate properties of (cubic) graphs which do not admit a LIEC with at most $2$ colors.
We present the results in Section~\ref{sec:2col}.

\section{Preliminaries}
\label{sec:prel}

In this section, we present terminology, notation, and auxiliary results that we are using in our proofs.

First, by a {\em $k$-LIEC} we refer to a locally irregular edge-coloring with at most $k$ colors.
For a $k$-LIEC $\sigma$ of $G$, we denote by $d_\sigma^i(u)$ the number of edges incident with a vertex $u$ and colored with $i$;
if the coloring $\sigma$ is clear from the context, we only write $d^i(u)$.
If two vertices are incident with the same number of edges of some color, we say that {\em they have the same color degree}.

A {\em $k$-vertex} (a {\em $k^+$-vertex}) is a vertex of degree $k$ (at least $k$).
A $k$-path $P=v_1,\dots, v_{k+1}$ is {\em pendant} in a graph $G$ if $d_G(v_1) = 1$, $d_G(v_{k+1}) \ge 2$,
and $d_G(v_i) = 2$ for $i=2,\dots,k$.
By a {\em $k$-thread} (a {\em $k^+$-thread}) in a graph $G$ we refer to a subgraph $T$ of $G$ isomorphic to the $k$-path (a $k^+$-path),
in which all vertices have degree $2$ also in $G$.

As described in~\cite{BauBenPrzWoz15}, decomposable graphs have a very specific structure.
In particular, the graphs which are not decomposable are odd paths, odd cycles, 
and graphs from the family $\mathcal{T}$ defined recursively as follows:
\begin{itemize}
	\item{} the triangle is in $\mathcal{T}$;
	\item{} every other graph in $\mathcal{T}$ is constructed by taking an 
			auxiliary graph $F$ which might either be an even path or an odd path with a triangle glued to one end, 
			then choosing a graph $G \in \mathcal{T}$ containing a triangle with at least one vertex $v$ of degree $2$, 
			and finally identifying $v$ with a vertex of degree $1$ in $F$.
\end{itemize}
Note that all graphs in $\mathcal{T}$ have maximum degree $3$ and so any graph $G$ with $\Delta(G) \ge 4$ is decomposable.

\section{Claw-free graphs}
\label{sec:claw}

In this section, we consider claw-free graphs with maximum degree $3$ 
and show that if such a graph is decomposable, then it admits a $3$-LIEC.

\begin{theorem}
	\label{thm:claw}
	For every decomposable claw-free graph $G$ with maximum degree $3$ it holds that
	$$
		\lir(G) \le 3\,.
	$$
\end{theorem}

\begin{proof}	
	We prove the theorem by contradiction.
	Let $G$ be a minimal counterexample to the theorem in terms of the number of edges;
	i.e.,
	$G$ is a connected decomposable subcubic claw-free graph with the minimum number of edges such that $\lir(G) > 3$.
	We first establish several structural properties of $G$.
	
	\begin{claim}
		\label{cl:pendant2path}
		$G$ does not contain a pendant $2$-path, i.e., in $G$, there is no edge $uv$ with $d(u) = 1$ and $d(v) = 2$.
	\end{claim}

	\begin{proofclaim}
		Since $G$ is decomposable, by the definition of $\mathcal{T}$, also the graph $G' = G \setminus \set{u,v}$ is decomposable.
		Thus, by the minimality, $G'$ admits a $3$-LIEC $\sigma'$ which induces a partial $3$-LIEC $\sigma$ of $G$
		with only the two edges incident with $v$ being non-colored.
		Let $w$ be the neighbor of $v$ distinct from $u$.
		Since $w$ is incident with at most two colored edges, there is a color $\alpha\in\{1,2,3\}$ such that $d^\alpha(w) = 0$,
		and so we can set $\sigma(vw) = \sigma(uv) = \alpha$, hence extending $\sigma$ to all edges of $G$, a contradiction.
	\end{proofclaim}

	\begin{claim}
		\label{cl:2thread3cyc}
		$G$ does not contain any $3$-cycle incident with two $2$-vertices.
	\end{claim}

	\begin{proofclaim}
		Suppose the contrary and let $C = v_1v_2v_3$ be a $3$-cycle with $d(v_1)=d(v_2)=2$.
		Clearly, $d(v_3)=3$; let $u$ be the neighbor of $v_3$, distinct from $v_1$ and $v_2$.
		By the construction of $\mathcal{T}$, 
		the graph $G' = G \setminus \set{v_1,v_2,v_3}$ is decomposable 
		and it admits a $3$-LIEC, which induces a partial $3$-LIEC $\sigma$ of $G$,
		where the edges $v_1v_2$, $v_2v_3$, $v_1v_3$, and $uv_3$ are non-colored.
		We may assume that $d^1(u)=0$.
		Then, we complete the coloring $\sigma$ by setting $\sigma(uv_3)=\sigma(v_1v_3)=1$ and $\sigma(v_1v_2)=\sigma(v_2v_3)=2$, a contradiction.
	\end{proofclaim}

	\begin{claim}
		\label{cl:2thread4cyc}
		$G$ does not contain any $4$-cycle incident with two consecutive $2$-vertices.
	\end{claim}

	\begin{proofclaim}
		Suppose the contrary and let $C = v_1v_2v_3v_4$ be a $4$-cycle with
		$d(v_1)=d(v_2)=2$.
		Since $\Delta(G)=3$, we also assume that $d(v_3)=3$, 
		and since $G$ is claw-free, $d(v_4) = 3$ and there is a vertex $u$
		adjacent to $v_3$ and $v_4$.
		
		Suppose that $G' = G \setminus \set{v_2,v_3}$ is decomposable.
		Then, $G'$ admits a $3$-LIEC which induces a partial $3$-LIEC $\sigma$ of $G$ 
		where only the edges $v_1v_2$, $v_2v_3$, $v_3v_4$, and $uv_3$ are non-colored.
		We may assume that $\sigma(v_1v_4)=\sigma(uv_4) = 1$ and $d^2(u)=0$.
		Hence, we can complete the coloring by coloring the edges 
		$v_1v_2$, $v_2v_3$, $v_3v_4$, and $uv_3$ with color $2$, a contradiction.
		
		If $G'$ is not decomposable, 
		then $G'' = G \setminus \set{v_1,v_2,v_3}$ is decomposable
		and it admits a $3$-LIEC which induces a partial $3$-LIEC $\sigma$ of $G$ 
		where the edges $v_1v_2$, $v_1v_4$, $v_2v_3$, $v_3v_4$, and $uv_3$ are non-colored.
		Now, we may assume that $\sigma(uv_4) = 1$ and $d^2(u)=0$.
		We complete the coloring by coloring the edges 
		$v_2v_3$, $v_3v_4$, and $uv_3$ with color $2$,
		and the edges $v_1v_2$, $v_1v_4$ with $3$, a contradiction. 
	\end{proofclaim}

	\begin{claim}
		\label{cl:2thread}
		$G$ does not contain any $2^+$-thread.
	\end{claim}
	
	\begin{proofclaim}
		Suppose the contrary and let $u$ and $v$ be adjacent $2$-vertices in $G$.
		In particular, without loss of generality, 
		we choose the pair in such a way that the second neighbor of $u$ is a $3$-vertex $w$.
		We label the vertices as depicted in Figure~\ref{fig:2thread}.
		By Claim~\ref{cl:2thread3cyc}, $w \ne z$, and
		since $G$ is claw-free, the two neighbors $w_1$ and $w_2$ of $w$, distinct from $u$, are adjacent.		
		We additionally assume that $d(w_1) \le d(w_2)$.
		\begin{figure}[htp!]
			$$
				\includegraphics{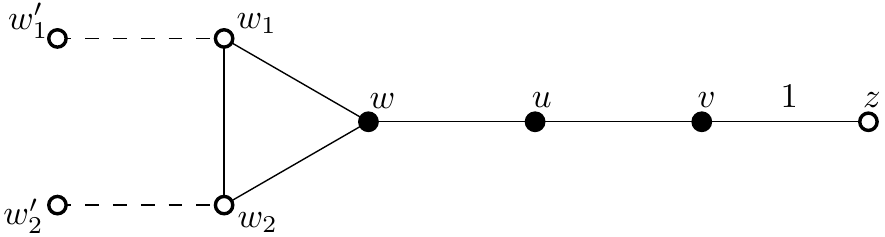}
			$$
			\caption{A $2^+$-thread in $G$ is reducible. 
				We depict vertices of nonspecified degrees as empty circles 
				and possibly nonexisting edges dashed, 
				where the vertices incident only to dashed edges may also not exist. 
				Note also that $w_1' = w_2'$ is possible.}
			\label{fig:2thread}
		\end{figure}
		
		Let $G'$ be the graph obtained from $G$ by contracting the edges $uv$ and $uw$.
		By Claim~\ref{cl:2thread4cyc}, $z$ is not adjacent to $w$, and consequently, $G'$ is a simple graph.
		Note that $\Delta(G') = 3$ and by construction of $\mathcal{T}$, $G'$ is also decomposable.		
		Moreover, by the minimality, $G'$ admits a $3$-LIEC $\sigma'$ which induces a partial $3$-LIEC $\sigma$ of $G$ 
		where only the edges $uv$ and $uw$ are non-colored.
		Without loss of generality, we may assume that $\sigma(vz)=1$.
		We show that $\sigma$ always extends to a $3$-LIEC of $G$ 
		by considering the cases regarding the colors of $ww_1$ and $ww_2$.
		
		\medskip
		\textbf{Case 1.}
		{\em The color of $vz$ does not appear at $w$, i.e., $d^{1}(w)=0$.} \quad
		Note that this case implies that at least one of $w_1$ and $w_2$ is a $3$-vertex, so $d(w_2) = 3$.
		If $\sigma(ww_1) = \sigma(ww_2)$, then we color the edges $uw$ and $uv$ with the color
		distinct from $\sigma(ww_1)$ and $\sigma(vz)$, and we are done.
		Thus, we may assume, without loss of generality, 
		that $\sigma(ww_1)=2$ and $\sigma(ww_2)=3$.
		Moreover, by the assumption that $d(w_2) = 3$ and the symmetry, we may also assume that $\sigma(w_1w_2) = 2$ 
		and thus $d^{3}(w_2)=2$.
		
		Now, we consider three subcases regarding the colors incident with $w_1$.
		First, if either $d(w_1) = 2$ or $d^2(w_1) = 2$ and $d^3(w_1) = 1$, 
		then we recolor $ww_1$ and $w_1w_2$ with $1$, and color $uw$ and $uv$ with $2$.
		Second, if $d^2(w_1) = 2$ and $d^1(w_1) = 1$, then there are two possibilities. 		
		If $d^1(w_1')=3$, then we set $\sigma(w_1w_2)=1$ and $\sigma(ww_1)=\sigma(uw)=\sigma(uv)=3$.
		If $d^1(w_1')=2$, then we proceed as in the first subcase.
		Third, if $d^2(w_1) = 3$, then we again consider the colors incident with $w_1'$. 
		If $d^2(w_1') = 1$, then we set $\sigma(ww_1)=\sigma(uw)=\sigma(uv)=3$.
		Otherwise, $d^2(w_1') = 2$ and we again color the edges as in the first subcase.
		
		\medskip
		\textbf{Case 2.}
		{\em $d^{1}(w)=1$.} \quad
		In this case, we set $\sigma(uv)=1$ and consider four subcases.
		
		\medskip
		\textbf{Case 2.1.}		
		{\em Suppose that $d(w_1) = 2$ and $\sigma(ww_1) = 1$.} \quad
		Then $d^1(w_1)=1$ since $\sigma$ is a LIEC of $G'$. We may therefore assume that $\sigma(ww_2)=\sigma(w_1w_2)=2$.
		Thus, setting $\sigma(ww_1) = \sigma(uw)=3$ extends $\sigma$ to all edges of $G$, a contradiction.
		
		\medskip
		\textbf{Case 2.2.}		
		{\em Suppose that $d(w_1) = 2$ and $\sigma(ww_2) = 1$.} \quad
		Then, we may assume that $\sigma(ww_1) = \sigma(w_1w_2)=2$ and $\sigma(w_2w_2')=3$ (in the case when $d(w_2)=2$, we proceed as in Case~2.1).
		Thus, we may set $\sigma(ww_2)=2$ and $\sigma(ww_1) = \sigma(uw)=3$, a contradiction.
		
		\medskip
		\textbf{Case 2.3.}		
		{\em Suppose that $d(w_1) = 3$ and $d^1(w_1)=3$.} \quad
		Note that, by the symmetry, this covers also the case with $d(w_1)=3$, $d^1(w_2)=3$.
		Next, we may assume that $\sigma(ww_2) = 2$.
		If $d^1(w_1') = 1$, then we set $\sigma(ww_1) = \sigma(uw) = 2$ (note that $\sigma(w_2w_2') = 2$, thus $d^1(w_2)\not = d^1(w_1)$).
		Otherwise, $d^1(w_1')=2$ and we set $\sigma(w_1w_2)=2$ and $\sigma(ww_1)=\sigma(ww_2)=\sigma(uw)=3$, a contradiction.
		
		\medskip
		\textbf{Case 2.4.}		
		{\em Suppose that $d(w_1) = 3$ and $d^1(w_1)=1$ with $\sigma(ww_1) = 1$.} \quad		
		Note that, by the symmetry, this covers also the case with $d(w_1)=3$, $d^1(w_2)=1$ and $\sigma(ww_2)=1$.
		Next, we may assume that $\sigma(ww_2) = 2$.		
		We consider three cases regarding the colors incident with $w_1$.
		First, if $d^2(w_1)=2$, then we set $\sigma(ww_1)=\sigma(uw)=3$.
		Second, if $d^3(w_1)=2$, then we set $\sigma(ww_1)=\sigma(uw)=2$.
		Third, suppose that $d^2(w_1)=d^3(w_1)=1$.
		Note that in this case $\sigma(w_1w_2)=2$ and $\sigma(w_1w_1')=3$, otherwise $\sigma$ is not a LIEC of $G'$.
		If $\sigma(w_2w_2')=1$ and $d^3(w_1') = 2$, then we set $\sigma(w_1w_2)=\sigma(ww_1)=3$ and $\sigma(uw)=2$.
		If $\sigma(w_2w_2')=1$ and $d^3(w_1') = 3$, then we set $\sigma(w_1w_2)=3$ and $\sigma(ww_1)=\sigma(uw)=2$.
		Next, if $\sigma(w_2w_2')=3$, then we set $\sigma(w_1w_2)=1$ and $\sigma(uw)=2$.
		Otherwise, $\sigma(w_2w_2')=2$.
		In that case, if $d^2(w_2')=2$, then we set $\sigma(w_1w_2)=\sigma(ww_2)=1$ and $\sigma(ww_1)=\sigma(uw)=2$.
		Therefore, $d^2(w_2') = 1$. 
		If $d^3(w_1') = 2$, then we set $\sigma(ww_1)=\sigma(w_1w_2)=\sigma(uw)=3$. 
		Otherwise $d^3(w_1') = 3$ and we set $\sigma(ww_1)=3$ and $\sigma(uw)=2$.
		Thus, $\sigma$ can be extended to all edges of $G$, a contradiction. 
		
		\medskip
		\textbf{Case 3.}
		{\em $d^{1}(w)=2$.} \quad
		By the symmetry, we may assume that $\sigma(w_1w_2) = \sigma(w_2w_2') = 2$.
		First, if $d^{1}(w_1)=2$ and $d^{1}(z)=1$, then we set $\sigma(ww_2)=\sigma(uw)=3$ and $\sigma(uv)=1$.
		Second, if $d^{1}(w_1)=1$ and $d^{1}(z)=2$, then we set $\sigma(uw) = \sigma(uv) = 2$.
		Third, if $d^{1}(w_1)=2$ and $d^{1}(z)=2$, then we can set $\sigma(w_1w_2) = 1$ and $\sigma(ww_2) = \sigma(ww_1) = \sigma(uw) = \sigma(uv) = 2$.
		Fourth, if $d^{1}(w_1)=1$ and $d^{1}(z)=1$, then we first set $\sigma(uv) = 1$. 
		Note that $\sigma(w_1w_1') = 3$ and we consider two cases with respect to $d^3(w_1')$.
		If $d^3(w_1') = 3$, then we set $\sigma(ww_1) = \sigma(ww_2) = \sigma(uw) = 3$.
		Otherwise, $d^3(w_1') = 2$ and we set $\sigma(ww_2) = 2$ and $\sigma(w_1w_2) = \sigma(ww_1) = \sigma(uw) = 3$.
		Thus, $\sigma$ can always be extended to all edges of $G$, a contradiction.
	\end{proofclaim}

	\begin{claim}
		\label{cl:bridge}
		$G$ does not have a non-trivial bridge.
	\end{claim}
	
	\begin{proofclaim}
		Suppose the contrary and let $uv$ be a bridge in $G$ with $d(u)>1$ and $d(v)>1$.
		The graph $G \setminus \set{uv}$ has exactly two connected components;
		let $G_u$ be the component containing $u$ and $G_v$ the component containing $v$.		
		By Claim~\ref{cl:2thread}, at least one of $u$ and $v$ is a $3$-vertex, say $d(u) = 3$.
		Moreover, Claim~\ref{cl:pendant2path} and the fact that $G$ is claw-free imply 
		that each of $G_u$ and $G_v$ is either isomorphic to a triangle (which is not possible due to Claim~\ref{cl:2thread3cyc}) or has maximum degree equal to $3$.

		Suppose first that $d(v)=2$ and let $w$ be the neighbor of $v$ distinct from $u$.
		Let $G_u' = G_u \cup \set{uv}$.
		By Claims~\ref{cl:2thread} and~\ref{cl:pendant2path}, $d(w)=3$.
		This means that $\Delta(G_u') = \Delta(G_v) = 3$. 
		By the construction of graphs in $\mathcal{T}$, neither $G_u'$ nor $G_v$ are in $\mathcal{T}$ 
		(no graph in $\mathcal{T}$ has a pendant edge with one endvertex of degree $3$) 
		and thus both are decomposable and claw-free.
		Therefore, by the minimality, $G_u'$ and $G_v$ admit $3$-LIEC $\sigma_u$ and $\sigma_v$, respectively.
		It is easy to see that a permutation of colors in $\sigma_v$ such that $\sigma_u(uv)\neq \sigma_v(vw)$ induces 
		a $3$-LIEC of $G$, a contradiction.
		
		So, we may assume that $d(v)=3$.
		Let $G_u' = G_u \cup \set{uv}$ and $G_v' = G_v \cup \set{uv}$.
		As deduced above, both $G_u'$ and $G_v'$ are decomposable (and claw-free).
		Therefore, if $G_u$ is decomposable, then there is a $3$-LIEC of $G$ induced by colorings of $G_u$ and $G_v'$,
		and similarly, if $G_v$ is decomposable, then there is a $3$-LIEC of $G$ induced by colorings of $G_u'$ and $G_v$.
		We may thus assume that neither $G_u$ nor $G_v$ are decomposable and consequently, both are in $\mathcal{T}$. 
		In that case, $G$ is also in $\mathcal{T}$ and thus non-decomposable, a contradiction.
	\end{proofclaim}

	\begin{claim}
		\label{cl:diamond}
		$G$ does not contain adjacent triangles.
	\end{claim}
	
	\begin{proofclaim}
		Suppose the contrary and let $uwz$ and $vwz$ be two adjacent triangles in $G$.
		If $d(u) = 2$ (and by the symmetry, if $d(v) = 2$), then, by Claim~\ref{cl:bridge}, 
		$G$ is a graph on at most $5$ vertices, which is trivially $2$-colorable.
		
		Thus, $d(u) = d(v) = 3$. 
		Let $u'$ and $v'$ be the neighbors of $u$ and $v$, respectively, both distinct from $w$ and $z$.
		If either $uu'$ or $vv'$ is a pendant edge, then, by Claim~\ref{cl:bridge}, 
		$G$ is a graph on at most $6$ vertices, which is $2$-colorable.
		Thus, we may assume that neither $uu'$ nor $vv'$ is a bridge.
		
		In the case when $u' = v'$, we have that $G$ is a $K_4$ with one edge subdivided (as $G$ is claw-free), which admits a $3$-LIEC.
		
		So we may assume that $u' \neq v'$.
		First observe that if $u'$ and $v'$ are adjacent, 
		then none of them is a $2$-vertex or incident with a pendant edge, 
		since $G$ would be a decomposable graph on at most $7$ vertices (in the case when $u'$ and $v'$ have a common neighbor) by Claim~\ref{cl:bridge}, 
		or not a claw-free graph.
		Therefore, $u'$ and $v'$ are not adjacent and none of them is adjacent to a $1$-vertex.
		Now, let $G'$ be the graph obtained from $G$ by removing the vertices $w$ and $z$, and identifying $u$ and $v$ into a single vertex $x$. 
		Note that $G'$ is claw-free and decomposable. Since $G$ is bridgeless, $G'$ is also bridgeless, which implies that there exists a path between $u'$ and $v'$ not containing the edges $uu'$ and $vv'$.
		Therefore, by the minimality, $G'$ admits a $3$-LIEC $\sigma'$ which induces a partial $3$-LIEC $\sigma$ of $G$,
		where $\sigma(uu')=\sigma'(xu')$ and $\sigma(vv')=\sigma'(xv')$.
		Without loss of generality, we consider two cases.
		First, suppose that $\sigma(uu') = \sigma(vv') = 1$.
		Then, we set $\sigma(uw) = \sigma(vw) = \sigma(wz) = 1$ and $\sigma(uz) = \sigma(vz) = 2$,
		hence extending $\sigma$ to $G$, a contradiction. 		
		Second, suppose that $\sigma(uu') = 1$ and $\sigma(vv') = 2$.
		If $d^1(u') = 3$, 
		then we set $\sigma(uw) = \sigma(vw) = \sigma(wz) = 1$ and $\sigma(uz) = \sigma(vz) = 3$.
		Thus, we may assume that $d^1(u') = 2$.
		Then, we set $\sigma(uw) = \sigma(uz) = \sigma(vw) = 1$ and $\sigma(wz) = \sigma(vz) = 3$.
		Thus, we can always extend $\sigma$ to all edges of $G$, a contradiction.
	\end{proofclaim}

	\begin{claim}
		\label{cl:widediamond}
		$G$ does not contain any $4$-cycle adjacent to two triangles.
	\end{claim}
	
	\begin{proofclaim}
		Suppose the contrary and let $uw_1z_1$ and $vw_2z_2$ be triangles in $G$ adjacent to the $4$-cycle $C = w_1w_2z_2z_1$.
		A similar reasoning as in Claim~\ref{cl:diamond} shows that either $G$ is a graph on at most $9$ vertices 
		admitting a $3$-LIEC or $d(u) = d(v) = 3$ with $uu', vv'\in E(G)$ 
		(where $u'$ and $v'$ are again the neighbors of $u$ and $v$, respectively, not incident with $C$), 
		$u' \neq v'$, and $u',v'$ are $2^+$-vertices.
		Moreover, by Claim~\ref{cl:bridge}, neither $uu'$ nor $vv'$ is a bridge, neither $u'$ nor $v'$ is adjacent to a $1$-vertex,
		and $u'$ is not adjacent to $v'$.
		
		Now, let $G'$ be the graph obtained from $G$ by removing the vertices of $C$, 
		and identifying $u$ and $v$ into a single vertex $x$. 
		By a similar reasoning as in Claim~\ref{cl:diamond} we have that $G'$ is claw-free and decomposable.
		Furthermore, by the minimality, $G'$ admits a $3$-LIEC $\sigma'$ which induces a partial $3$-LIEC $\sigma$ of $G$,
		where $\sigma(uu')=\sigma'(xu')$ and $\sigma(vv')=\sigma'(xv')$.
		Without loss of generality, we again consider two cases.
		First, suppose that $\sigma(uu') = \sigma(vv') = 1$.
		Then, we set $\sigma(uw_1) = \sigma(w_1w_2) = \sigma(w_1z_1) = \sigma(vz_2) = 1$, $\sigma(uz_1) = \sigma(z_1z_2) = 2$, 
		and $\sigma(vw_2) = \sigma(w_2z_2) = 3$, hence extending $\sigma$ to $G$, a contradiction.
		Second, suppose that $\sigma(uu') = 1$ and $\sigma(vv') = 2$.
		Then, we set $\sigma(uw_1) = \sigma(uz_1) = \sigma(vw_2) = \sigma(vz_2) = 3$, $\sigma(w_1w_2) = \sigma(w_1z_1) = 1$, 
		and $\sigma(w_2z_2) = \sigma(z_1z_2) = 2$.
		Again, we extended $\sigma$ to all edges of $G$, a contradiction.
	\end{proofclaim}

	\begin{claim}
		\label{cl:trian2vert}
		$G$ does not contain a triangle incident with a $2$-vertex or incident with a pendant edge.
	\end{claim}
	
	\begin{proofclaim}
		Suppose to the contrary that $uvw$ is a triangle in $G$ with 
		either $d(u)=2$ or $u$ having a neighbor $u'$ distinct from $v$ and $w$ such that $d(u')=1$.
		If also $d(v)=2$ or $v$ is adjacent to a pendant vertex, then the same holds for $w$ by Claim~\ref{cl:bridge}
		(but satisfying the requirement that $G$ is decomposable).
		In all these cases, it is easy to verify that $G$ is a graph on at most $6$ vertices which has a $3$-LIEC.

		Therefore, we may assume that $d(v)=d(w)=3$. 
		Let $v'$ and $w'$ be the third neighbors of $v$ and $w$, respectively.
		Since $G$ is bridgeless, so is $G' = G \setminus \set{u}$ (if $u'$ exists we remove it as well) and thus $G'$ is decomposable.
		Observe also that $G'$ is claw-free and therefore, by the minimality, $G'$ admits a $3$-LIEC $\sigma'$,
		which induces a partial $3$-LIEC $\sigma$ of $G$ with only the edges incident with $u$ being non-colored.
		Without loss of generality, we may assume that $\sigma(vv')=\sigma(vw)=1$ and $\sigma(ww')=2$.
		We complete the coloring of $G$ by coloring the edges (two or three if $u'$ exists) incident with $u$ by color $3$,
		a contradiction.		
	\end{proofclaim}		
	
	From Claims~\ref{cl:pendant2path} and~\ref{cl:trian2vert} we also infer that $G$ has minimum degree at least $2$.

	\begin{claim}
		\label{cl:triandist}
		Every pair of adjacent $3$-vertices in $G$ has a common neighbor,
		i.e., no two adjacent vertices are incident with distinct triangles. 
	\end{claim}

	\begin{proofclaim}
		Suppose the contrary and let $uvw$ and $xyz$ be triangles in $G$
		with $w$ and $z$ being adjacent. 
		We adopt the labeling of the vertices as depicted in Figure~\ref{fig:adj-trian}.
		Note that by Claims~\ref{cl:diamond} and~\ref{cl:trian2vert}, $u'$ and $v'$ exist and are distinct, 
		and similarly, $x'$ and $y'$ exist and are distinct.
		\begin{figure}[htp!]
			$$
				\includegraphics{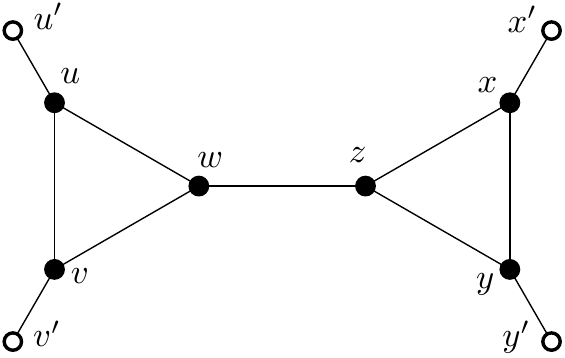}
			$$
			\caption{Two connected triangles do not appear in $G$.}
			\label{fig:adj-trian}
		\end{figure}		
		
		Now, consider the graph $G' = G \setminus \set{w}$.
		Since $G$ is bridgeless, $G'$ is connected.
		Moreover, there is a path $P$ between $u'$ and $v'$ containing neither $uu'$ nor $vv'$ in $G$ and thus $P$ is also contained in $G'$.
		This means that $G'$ contains a cycle of length more than $3$, and therefore $G'$ is decomposable.
		Thus, by the minimality, $G'$ admits a $3$-LIEC $\sigma'$,
		which induces a partial $3$-LIEC $\sigma$ of $G$ with only the edges incident with $w$ being non-colored.
		We will show that we can always extend $\sigma$ to all edges of $G$.
		
		Without loss of generality, we may assume that $\sigma(uu') = \sigma(uv) = 1$ and $\sigma(vv')=2$.
		If $d^3(z) = 0$, then we color the edges incident with $w$ by color $3$ and we are done.
		So, we may also assume that $\sigma(xz)=3$.
		Now we consider cases regarding the color of $yz$.
		
		Suppose first that $\sigma(yz)=3$.
		Then, $d^3(x)=d^3(y)=1$ and we set $\sigma(wz)=3$.
		If $d^1(u')=1$, then we set $\sigma(uw)=1$ and $\sigma(vw)=3$, a contradiction.
		So, $d^1(u')=3$ and we consider the colors around $v'$.
		If $d^2(v')=2$, then we set $\sigma(uw)=3$ and $\sigma(uv)=\sigma(vw)=2$, a contradiction.
		Finally, if $d^2(v')=3$, then we set $\sigma(uw)=1$, $\sigma(uv)=2$, and $\sigma(vw)=3$, a contradiction.
		
		Second, suppose that $\sigma(yz)=1$.
		We set $\sigma(wz)=\sigma(vw)=2$ and consider two subcases.
		If $d^2(v')=3$, then we set $\sigma(uw)=2$.
		Otherwise, if $d^2(v')=2$, then set $\sigma(uw)=1$ and set $\sigma(uv)=2$, a contradiction.
		
		Third, suppose that $\sigma(yz)=2$.
		If $d^3(x) = 3$, then we color all edges incident with $w$ by $3$ and we are done.		
		Similarly, if $d^2(y)=3$, then we set $\sigma(wz)=2$ and $\sigma(uw)=\sigma(vw)=3$, a contradiction.
		Therefore, $d^3(x) = 2$ and $d^2(y)=2$. 		
		Suppose first that $\sigma(xy)=3$. 
		Then, $\sigma(yy')=2$. 
		If $\sigma(xx')=1$, then we set $\sigma(xy)=2$ and $\sigma(yz)=3$, which brings us to the setting resolved in the first case, a contradiction.
		If $\sigma(xx')=2$, then we recolor $xy$ and $xz$ to $1$, and color the edges incident with $w$ by $3$, a contradiction.
		
		Therefore, we may assume that $\sigma(xy)=2$. 
		If $\sigma(yy')=1$, then we set $\sigma(xy)=3$, $\sigma(xz)=2$, and color the edges incident with $w$ by $3$, a contradiction.
		Finally, if $\sigma(yy')=3$, we recolor $xy$ and $yz$ to $1$
		and consider two subcases regarding the colors at $v'$.
		If $d^2(v')=3$, then we color the edges incident with $w$ by $2$, a contradiction.
		If $d^2(v')=2$, then we set $\sigma(uw)=1$, $\sigma(uv)=2$, and color the non-colored edges incident with $w$ by $2$, a contradiction.
	\end{proofclaim}	
	
	From the above properties of $G$, it follows that every $2$-vertex in $G$ has two neighbors of degree $3$,
	every triangle is incident only with $3$-vertices, and every pair of adjacent $3$-vertices is incident with the same triangle.
	This means that by contracting every triangle in $G$ to a single vertex,
	we obtain a bipartite graph $H$ with partition $(A,B)$, 
	where all vertices in $A$ have degree $3$ and all vertices in $B$ have degree $2$.
	By K\"{o}nig's Theorem, $H$ admits a proper $3$-edge-coloring $\sigma_H$,
	which induces a partial $3$-LIEC coloring $\sigma$ of $G$ with the edges of triangles being non-colored.
	Note that every $2$-vertex in $G$ is incident with edges of distinct colors, 
	and observe also that in every triangle $T=uvw$, the three edges incident with the vertices $u$, $v$, and $w$
	not belonging to $T$ have distinct colors, say $1$, $2$, and $3$, respectively.
	Thus, we can color the edges of $T$ by setting $\sigma(uv)=1$, $\sigma(vw) = 2$, and $\sigma(uw)=3$.
	By coloring every triangle in such a way, we can thus extend the coloring $\sigma$ to all edges of $G$, a contradiction.
	This completes the proof.
\end{proof}

\section{Cycle permutation graphs and generalized Petersen graphs}
\label{sec:Petersen}

In this section, we deal with two well-known families of cubic graphs; namely,
cycle permutation graphs and generalized Petersen graphs.
We begin with definitions.

Let $G$ and $H$ be disjoint graphs on $n$ vertices for $n \geq 3$, 
with $V(G) = \{v_0, v_1,\dots, v_{n-1}\}$ 
and $V(H) = \{u_0, u_1,\dots, u_{n-1}\}$. 
Let $\phi \, : \, V(G) \rightarrow V(H)$ be a bijection between the vertices of $G$ and $H$. 
The {\em generalized permutation graph} $P_{\phi}(G,H)$ is a graph with the vertex set $V(G) \cup V(H)$ 
and the edge set $E(G) \cup E(H) \cup \set{v_i \phi(v_i) \mid \mathrm{for~} i \in \set{0,\dots,n-1}}$. 
In the case $G = H = C_n$, $P_{\phi}(C_n,C_n)$ is called a {\em cycle permutation graph}.
Note that cycle permutation graphs are cubic.

The {\em generalized Petersen graph} $P(n,k)$, with $k < \frac{n}{2}$, is a cubic graph 
with the vertex set $V(P(n,k)) = \{u_0, u_1, \dots,$ $u_{n-1}, v_0, v_1, \dots, v_{n-1}\}$ 
and the edge set $E(P(n,k)) = \{u_i u_{i+1}, u_i v_i, v_i v_{i+k} \ | \ 0 \leq i \leq n-1\}$ 
(indices are taken modulo $n$).

Let $\mathcal{R}_n^2$ be the family of all simple $2$-regular graphs on $n$ vertices.
For any $R \in \mathcal{R}_n^2$ and any bijection $\phi \, : \, V(C_n) \rightarrow V(R)$,
we call $P_\phi(C_n, R)$ a {\em ring permutation graph}.
Note that the union of all cycle permutation graphs and all generalized Petersen graphs 
is a subset of ring permutation graphs.

The main result of this section is the following.
\begin{theorem}
	\label{thm:ring}
	For any ring permutation graph $G$ it holds that $\lir(G) \le 3$.
\end{theorem}

Before we prove Theorem~\ref{thm:ring}, we introduce some additional notation and establish three auxiliary properties.

A {\em shrub} is a tree rooted at a vertex $r$ such that $d(r) = 1$; the only neighbor of the root $r$ is denoted $r^+$.
An edge coloring $\sigma$ of a shrub $T$ is an {\em almost locally irregular $2$-edge-coloring} (or a {\em $2$-ALIEC} for short) 
if it is a $2$-LIEC of $G$ 
or the edge $rr^+$ is the only edge of color $\sigma(rr^+)$ incident with $r^+$ and $\sigma$ restricted to the edges of $G - r$ is a $2$-LIEC of $G-r$.
We already mentioned that in~\cite{BauBenPrzWoz15} it was shown that every decomposable tree $T$ has $\chi_{irr}(T)\le 3$. 
In~\cite{BauBenSop15}, trees were investigated further, and the authors showed that every shrub admits a $2$-ALIEC. 
We will use this fact to prove the following lemma 
(note that the fact also follows from the proof of Lemma 3.1, the case $p=2$, in~\cite{BauBenSop15}).
\begin{lemma} 
	\label{lem:tree1}
	If a tree $T$ contains an edge $uv$ with $d(u) = 1$ and $d(v) = 3$, 
	then $\lir(T) \le 2$.
\end{lemma}

\begin{proof}
	Suppose to the contrary that $T$ does not admit a $2$-LIEC.	
	Let $T$ be a shrub with the root $u$, and let $v_1$, $v_2$ be the neighbors of $v$ distinct from $u$. 
	By~\cite[Theorem~3.2]{BauBenSop15}, $T$ admits a $2$-ALIEC $\sigma$ with colors from $\set{1,2}$. 
	Without loss of generality, $\sigma(uv) = 1$, $\sigma(vv_1) = \sigma(vv_2) = 2$, and $\sigma$ induces a $2$-LIEC of $T' = T-u$. 
	We will show that we can modify $\sigma$ to obtain a $2$-LIEC of the whole shrub $T$.

	Let $T_1$ be the component of $T - v$ containing $v_1$ to which we add the edge $vv_1$,
	and similarly, let $T_2$ be the component of $T - v$ containing $v_2$ to which we add the edge $vv_2$.
	Clearly, we cannot recolor $uv$ with color $2$, as that would imply that $T$ admits a $2$-LIEC.
	Therefore, at least one from $v_1$ and $v_2$, say $v_1$, is incident with three edges of color $2$. 
	Now, we swap the colors $1$ and $2$ in $T_1$ to obtain a coloring $\sigma'$.
	Note that either $\sigma'$ is a $2$-LIEC of $T$ or $d_{\sigma'}^2(v_2) = 1$.
	But in the latter case, we recolor $uv$ with color $2$ and we are done.
\end{proof}

\begin{lemma}
	\label{lem:tree2}
	If a tree $T$ contains a pendant $(2k-1)$-path, then $\lir(T) \le 2$.
\end{lemma}

\begin{proof}
	Let $T$ be a tree with a pendant $(2k-1)$-path $P=v_1,\dots, v_{2k}$. 
	We proceed by induction on the length of $P$. 
	If $k=1$, then $T$ contains an edge $v_1v_2$ with $d(v_1) = 1$ and $d(v_2) = 3$, 
	and hence it admits a $2$-LIEC by \cref{lem:tree1}. 
	So suppose that $k \geq 2$. 
	In this case, we can split $T$ into the shrub $T' = T - \set{v_1,\dots,v_{2k-2}}$ and the path $P'=v_1,\dots,v_{2k-1}$.
	Both of them admit a $2$-LIEC: $T'$ by \cref{lem:tree1} and $P'$ since it has an even length.
	Since the two colorings can easily be combined into a $2$-LIEC of $T$ 
	(by setting distinct colors for $v_{2k-2} v_{2k-1}$ and $v_{2k-1}v_{2k}$),
	we are done.
\end{proof}

\begin{lemma} 
	\label{lem:tree3}
	If a tree $T$ contains a $(2k+1)$-thread incident with two vertices of degree $3$, then $\lir(T) \le 2$.
\end{lemma}

\begin{proof}
	Let $T$ be a tree with a $(2k+1)$-thread $v_1,\dots,v_{2k+1}$ 
	such that $v_1$ has a $3$-neighbor $u_1$ and $v_{2k+1}$ has a $3$-neighbor $u_2$. 
	We split $T$ into two shrubs $T_1$ and $T_2$,
	where $T_1$ is the component of $T-v_2$ containing $u_1$
	and $T_2$ is the tree induced by the edges in $E(T) \setminus E(T_1)$.	
	Both shrubs admit a $2$-LIEC: $T_1$ by \cref{lem:tree1} and $T_2$ by \cref{lem:tree2}. 
	Combining the two colorings 
	(by setting distinct colors for $u_1v_1$ and $v_1v_2$), 
	we obtain a $2$-LIEC of $T$.
\end{proof}

Now, we are ready to prove the theorem.
\begin{proof}[Proof of Theorem~\ref{thm:ring}]
	Let $G = P_\phi(C_n,R)$ be a ring permutation graph.
	Label the vertices of the cycle $C_n$ (consecutively) with $v_1,\dots,v_{n}$.
	Let $R$ be composed of $k$ cycles, denoted $C_{i,\ell_i}$ for $i \in \set{1,\dots,k}$, where $\ell_i$ denotes the length of the cycle,
	and label the vertices of the cycle $C_{i,\ell_i}$ (consecutively) with $v_{i,1},\dots,v_{i,\ell_i}$.
	Note that $\sum_{i=1}^k \ell_i = n$.

	Now, consider the spanning subgraph $S$ of $G$, with the edge set defined such that for every $i \in \set{1,\dots,k}$:
	\begin{itemize}
		\item[$(a)$] if $\ell_i$ is even, 
					 then $E(C_{i,\ell_i}) \subset E(S)$ and $\phi^{-1}(v_{i,2j})v_{i,2j} \in E(S)$ for $j = 1,\dots,\frac{\ell_i}{2}$;		
		\item[$(b)$] if $\ell_i$ is odd, 
					 then $E(C_{i,\ell_i})\setminus \set{v_{i,1}v_{i,\ell_i}} \subset E(S)$, 
					 $\phi^{-1}(v_{i,2j})v_{i,2j} \in E(S)$ for $j = 1,\dots,\frac{\ell_i-1}{2}$, 
					 and $\phi^{-1}(v_{i,\ell_i})v_{i,\ell_i} \in E(S)$.
	\end{itemize}

	(Note that $S$ may contain some isolated vertices.) The vertices of $C_n$ can therefore be divided into three sets:
	\begin{itemize}
		\item{} $X_1 = \set{v \in V(C_n) \ | \ d_S(v) = 0}$, i.e.,
			the vertices that are not incident with any edge of $S$.
		
		\item{} $X_2 = \set{v \in V(C_n) \ | \ \exists i \ : \ \ell_i \textrm{ is odd and } v = \phi^{-1}(v_{i,\ell_i})}$, i.e.,
			the vertices whose neighbors in $S$ have degree $2$.

		\item{} $X_3 = V(C_n) \setminus (X_1 \cup X_2)$, i.e.,
			the vertices whose neighbors in $S$ have degree $3$.
	\end{itemize}
	Note that $X_1$ and $X_3$ are non-empty, and $|X_2|$ is equal to the number of odd cycles in $R$. 
	Moreover, note that $S$ is locally irregular.

	Suppose now that there is an edge $uv$ of the cycle $C_n$ such that $u \in X_1$ and $v \in X_3$.
	Then the graph $S' = S \cup \set{uv}$ is also locally irregular, and so $\lir(S') = 1$. 
	Otherwise, if there is no edge in $C_n$ with endvertices from $X_1$ and $X_3$,
	then there is an edge $e = uv$ of the cycle $C_n$ such that $u \in X_2$ and $v \in X_3$.
	Let $u = \phi^{-1}(v_{i,\ell_i})$ for some $i$ (note that $\ell_i$ is odd).
	Now, replace the edge $uv_{i,\ell_i}$ in $S$ with the edge $\phi^{-1}(v_{i,1})v_{i,1}$.
	In this way, we have that $u \in X_1$ and $v \in X_3$.
	Thus, the graph $S' = S \cup \set{uv}$ is again locally irregular and thus $\lir(S') = 1$.
	
	We finalize the proof by considering the graph $T' = G \setminus E(S')$ and showing that it admits a $2$-LIEC.
	It is easy to see that $T'$ is a tree. 
	If any cycle in $R$ has length at least $4$, 
	then $T'$ contains an edge with endvertices of degree $1$ and $3$, and so, by \cref{lem:tree1}, it admits a $2$-LIEC,
	implying that $G$ admits a $3$-LIEC.
	Thus, we may assume that all cycles in $R$ have length $3$, 
	meaning that $n=3k$ and $|X_1|=|X_2|=|X_3|=k$ (recall that the sets are defined in the graph $S$).
	Since the edge $uv$ from $C_n$ is added to $S'$, 
	we have that $|E(T') \cap E(C_n)| = 3k - 1 \equiv k-1 \bmod 2$
	and the number of $3$-vertices in $T'$ is $k-1$ (all of them belong to $X_1$).	
	
	If in $T'$ there is a $3$-vertex incident with a pendant $(2t+1)$-path, for some positive integer $t$, 
	then, by \cref{lem:tree2}, $T'$ admits a $2$-LIEC. 
	If in $T'$ there are two $3$-vertices joined by a $(2t+1)$-thread (i.e., a $(2t+2)$-path), 
	then, by \cref{lem:tree3}, $T'$ also admits a $2$-LIEC. 
	Otherwise, every pendant path has even length and every path between two $3$-vertices has odd length.
	Note that on every pendant path, there are precisely two edges not belonging to $C_n$,
	and thus even number of edges from $C_n$ belong to pendant paths.
	This means, since there are $k-2$ (odd) paths between $3$-vertices of $T'$,
	that $|E(T') \cap E(C_n)| \equiv (k-2) \bmod 2$, 
	a contradiction.
	This completes the proof.
\end{proof}

\section{Remarks on $2$-colorability}
\label{sec:2col}

In this section we focus on graphs which do not admit a $2$-LIEC. 
First, let us consider graphs with minimum degree $1$. 
Baudon et al.~\cite{BauBenPrzWoz15} showed that there are infinitely many trees which require $3$ colors for a locally irregular edge-coloring. 
On the other hand, all trees with $\Delta(G) \geq 5$ admit a $2$-LIEC~\cite[Theorem~3.3]{BauBenSop15}. 

Similarly, for graphs with minimum degree $2$, there also exists an infinite family of graphs which do not admit a $2$-LIEC~\cite{Sea19}. 
For example, the graph $H$ depicted in Figure~\ref{fig:4t+1} has $\lir(H) = 3$; 
in fact, every graph $G$ composed of two adjacent vertices $u$ and $v$ additionally connected with $k \ge 2$ paths of lengths $4t + 1$ for any $t \ge 1$
has $\lir(G) = 3$.
Note that graphs in the described family can have arbitrarily high maximum degree and arbitrarily high girth.
Moreover, they are all bipartite and thus the ones with even degrees of $u$ and $v$ provide examples 
for a negative answer to Question~1 asked in~\cite{LuzPrzSot18}.
\begin{figure}[htp!]
	$$
		\includegraphics{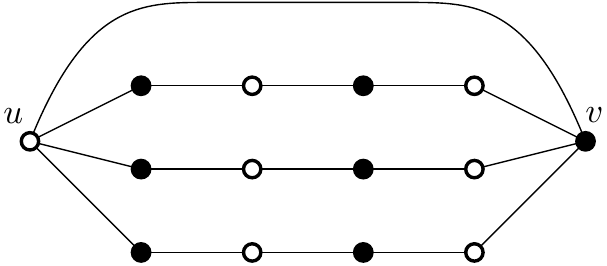}
	$$	
	\caption{The bipartite graph $H$ with minimum degree $2$ and $\lir(H)=3$.}
	\label{fig:4t+1}
\end{figure}

In the remainder of this section, we focus on graphs with minimum degree $3$; in particular, cubic graphs.

A result on the detection number of graphs due to Havet et al.~\cite{HavParSam14} 
implies that deciding whether a cubic graph admits a $2$-LIEC is NP-complete.
Our goal here is to determine properties of cubic graphs allowing the $2$-colorability.
For example, it is known that every regular bipartite graph with minimum degree at least $3$
(and hence every cubic bipartite graph)
admits a $2$-LIEC~\cite{BauBenPrzWoz15,HavParSam14}. 
On the other hand, not even all generalized Petersen graphs are such, e.g., $\lir(P(7,2)) = 3$. 

One may thus naturally ask which are the conditions causing a cubic graph not to admit a $2$-LIEC.
The following proposition asserts that, e.g., adjacent diamonds in a graph are one of them.
\begin{proposition}
	\label{prop:diam}
	If a cubic graph $G$ contains two diamonds connected with an edge as a subgraph, 
	then $\lir(G) \ge 3$.
\end{proposition}
\begin{figure}[htp!]
	$$
		\includegraphics{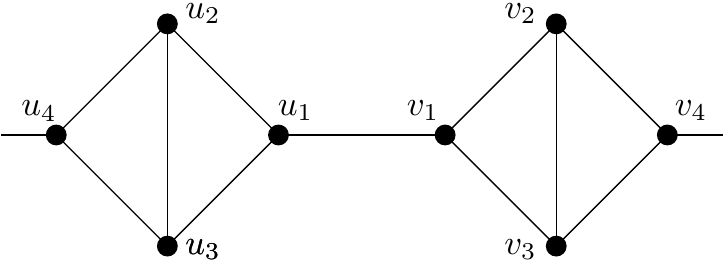}
	$$
	\caption{A cubic graph with two diamonds connected with an edge does not admit a $2$-LIEC.}
	\label{fig:double}
\end{figure}

\begin{proof}
	Suppose the contrary and let $G$ be a cubic graph containing two adjacent diamonds 
	(with the vertices labeled as depicted in Figure~\ref{fig:double})
	and let $\sigma$ be a $2$-LIEC of $G$ with the colors from $\set{1,2}$.
	It is easy to see that $G$ does not contain any monochromatic $3$-cycle, 
	otherwise at least two vertices of this $3$-cycle would have the same color degree in $\sigma$.

	Now, without loss of generality, we may assume that $\sigma(u_1v_1)=1$. 
	If $\sigma(v_1v_2)=\sigma(v_1v_3)=2$, then $\sigma(v_2v_3)=1$, otherwise we obtain a monochromatic $3$-cycle.  
	Moreover, exactly one of the edges $v_2v_4$ and $v_3v_4$ has color $1$, otherwise $d^1(v_2) = d^1(v_3)$.
	Without loss of generality let $\sigma(v_2v_4)=1$ and $\sigma(v_3v_4)=2$, but then $d^2(v_1) = d^2(v_3)$, a contradiction. 

	Therefore, at least one of the edges $v_1v_2$ and $v_1v_3$ has color $1$ 
	and, by the symmetry, at least one of the edges $u_1u_2$ and $u_1u_3$ has color $1$.	
	Since $\sigma$ is a $2$-LIEC, one of the vertices $u_1$ and $v_1$ is incident with two edges of color $1$ and the other 
	with three edges of color $1$, 
	say $d^1(v_1) = 2$ and $d^1(u_1) = 3$. 
	Then $\sigma(u_2u_3)= 2$ and, by the same reasoning as above, 
	$\sigma(u_2u_4) \neq \sigma(u_3u_4)$. 
	However, regardless of the color used on the edge incident with $u_4$ distinct from $u_2u_4$ and $u_3u_4$, 
	we obtain two adjacent vertices with the same color degree, a contradiction.
\end{proof}

In the proof of the NP-completness of deciding whether a cubic graph admits a $2$-LIEC,
many 3-cycles (even diamonds) are used. 
So, one may ask what happens if short cycles, particularly triangles, are forbidden.
The next theorem shows that there exists an infinite family of (cycle permutation) graphs with girth at least $4$ which do not admit a $2$-LIEC. 

Before stating the theorem, we define the graphs in the above-mentioned family.
The graph $\mathrm{XI}_n$ is a graph comprised of 
two cycles $v_1,\dots,v_{3n}$ and $u_1,\dots,u_{3n}$ with additional edges 
$v_{3i}u_{3i+1}$, $v_{3i+1}u_{3i}$, and $v_{3i+2}u_{3i+2}$ for all $i \in \{0,\dots,n-1\}$ (indices taken modulo $3n$). 
Note that $g(\mathrm{XI}_n) = 4$ for all $n \geq 2$.

\begin{theorem}
	\label{thm:theoremXI}
	For every positive integer $k$, we have
	$\lir(\mathrm{XI}_{2k}) = 2$ and $\lir(\mathrm{XI}_{2k+1}) = 3$. 
\end{theorem} 
 
\begin{proof}
	The graph $\mathrm{XI}_n$ consists of $n$ subgraphs induced on the vertices 
	$v_{3i}$, $u_{3i}$, $v_{3i+1}$, $u_{3i+1}$, $v_{3i+2}$, and $u_{3i+2}$ 
	with four half-edges incident with the vertices $v_{3i}$, $u_{3i}$, $v_{3i+2}$, and $u_{3i+2}$, for $i \in \{0,\dots,n-1\}$.
	We call each such configuration an XI-subgraph of $\mathrm{XI}_n$.
	
	Next, we consider all possible $2$-LIECs of an XI-subgraph (for clarity, we use colors from $\set{a,b}$).
	To the vertices $v_{3i}$, $u_{3i}$, $v_{3i+2}$, and $u_{3i+2}$ 
	we assign values $c_k$ such that $c \in \{a,b\}$ is the color used on the half-edge incident with the corresponding vertex 
	and $k\in\{1,2,3\}$ is the number of edges of color $c$ incident with this vertex (including the half-edge). 
	This way, we can encode every $2$-LIEC of an XI-subgraph with a quadruple $(p,q,r,s)$ 
	in which the components correspond to the values $c_k$ of the vertices $v_{3i}$, $u_{3i}$, $v_{3i+2}$, and $u_{3i+2}$, respectively, 
	where $p,q,r,s \in \{a_1,a_2,a_3,b_1,b_2,b_3\}$. 
	For example, the XI-subgraph in Figure~\ref{XI} has a $2$-LIEC coloring with the code $(a_3, a_3, a_1, b_3)$.
	
	\begin{figure}[htp!]
		$$
			\includegraphics{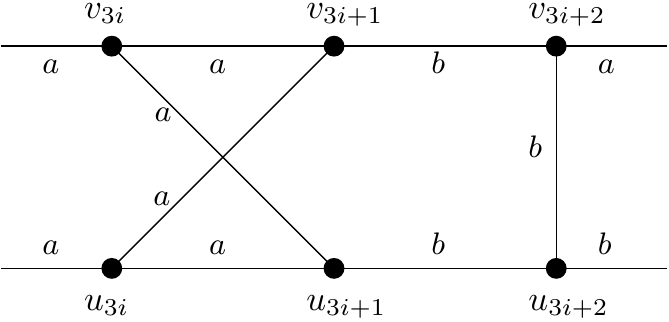}
		$$
		\caption{An XI-subgraph of $\mathrm{XI}_n$ with a $2$-LIEC with the code $(a_3,a_3, a_1, b_3)$.}
		\label{XI}
	\end{figure}

	Let $R_{\mathrm{XI}}$ denote the set of all codes of $2$-LIECs of an XI-subgraph. 
	We use the symmetries of the XI-subgraph to prove the following.
	\begin{claim}
		\label{claim_vymena}
		If $(p,q,r,s) \in R_{\mathrm{XI}}$, then also 
		$\set{(p,q,s,r),(q,p,r,s),(q,p,s,r)} \subset R_{\mathrm{XI}}$. 
	\end{claim}

	\begin{proofclaim} 
		It is easy to see that $(p,q,s,r)\in R_{\mathrm{XI}}$,
		since there exists an isomorphism from the $2$-LIEC $(p,q,r,s)$ to the $2$-LIEC $(p,q,s,r)$ by swapping 
		the vertices $v_{3i+1}$, $u_{3i+1}$ and the vertices $v_{3i+2}$, $u_{3i+2}$. 		
		Similarly, there exists an isomorphism from the $2$-LIEC $(p,q,r,s)$ to the $2$-LIEC $(q,p,r,s)$ by swapping
		the vertices $v_{3i}$, $u_{3i}$. 	
		Finally, by composing the above two isomorphisms, 
		we infer that $(q,p,s,r)\in R_{\mathrm{XI}}$.
	\end{proofclaim}

	We proceed by noting that $(a_3,a_2,r,s) \notin R_{\mathrm{XI}}$ for arbitrary $r,s \in \{a_1,a_2,a_3,b_1,b_2,b_3\}$. 
	On the other hand, $(a_2,a_2,a_3,b_1),(a_2,a_2,a_3,a_2)\in R_{\mathrm{XI}}$, 
	hence for some pairs $\{p,q\}$ there exist more than one suitable pairs $\{r,s\}$. 

	By Claim~\ref{claim_vymena}, it suffices to consider only the codes $(p, q, r, s)$, 
	in which $p \geq q$ and $r \geq s$ 
	(using the ordering $a_3 \geq a_2 \geq a_1 \geq b_1 \geq b_2 \geq b_3$). 
	We call such codes {\em ordered codes}. There are 24 ordered codes in $R_{\mathrm{XI}}$:	
	\begin{alignat*}{6}
		(a_3,a_3,a_1,b_3), &\ (a_2,a_2,a_3,b_1), &\ (a_2,a_2,a_2,a_1), &\ (a_2,a_2,a_2,b_2), &\ (a_2,a_2,b_1,b_2), &\ (a_1,a_1,a_2,b_2), \\
		(a_3,b_2,a_3,b_2), &\ (a_3,b_2,b_2,b_3), &\ (a_3,b_3,a_3,a_2), &\ (a_3,b_3,b_2,b_3), &\ (a_2,b_1,a_2,b_3), &\ (a_2,b_1,b_2,b_3), \\
		(a_2,b_2,a_2,a_1), &\ (a_2,b_2,b_1,b_2), &\ (a_2,b_3,a_3,a_2), &\ (a_2,b_3,a_2,b_3), &\ (a_1,b_2,a_3,a_2), &\ (a_1,b_2,a_3,b_2), \\
		(b_1,b_1,a_2,b_2), &\ (b_2,b_2,a_2,a_1), &\ (b_2,b_2,a_2,b_2), &\ (b_2,b_2,a_1,b_3), &\ (b_2,b_2,b_1,b_2), &\ (b_3,b_3,a_3,b_1).
	\end{alignat*}

	In a $2$-LIEC of $\mathrm{XI}_n$, codes of two consecutive XI-subgraphs (i.e., subgraphs connected by two edges represented by the four half-edges) must be carefully combined.
	Namely, the adjacent half-edges must have the same color, but their endvertices must have distinct color degree.	
	For example, the XI-subgraph colored $(a_3,a_3,a_1,b_3)$ can be connected to 
	the XI-subgraph colored $(a_2,b_2,a_2,a_1)$. 
	In fact, in this case, they can also be connected in the reverse order, implying that we can color any even number of XI-subgraphs, 
	and thus $\lir(\mathrm{XI}_{2k})=2$ for $k \geq 1$. 

	Next, note that if the half-edges incident with $v_{3i}$ and $u_{3i}$ have the same color $c$, 
	then $d^c(v_{3i}) = d^c(u_{3i})$. 
	Moreover, there is no code $(p,q, a_3, a_1)$, $(p,q, b_1, b_3)$, or $(p,q,r,r)$ in $R_{\mathrm{XI}}$ 
	for any $r\in \{a_1,a_2,a_3,b_1,b_2,b_3\}$. 
	Hence, if an XI-subgraph is colored $(a_2,a_2,r,s)$ or $(b_2,b_2,r,s)$, 
	then there is no coloring for the previous XI-subgraph that could be combined with the two codes above.
	Thus, colorings of XI-subgraphs with codes of types $(a_2,a_2,r,s)$ and $(b_2,b_2,r,s)$ are not used
	in any $2$-LIEC of $\mathrm{XI}_n$.

	We thus have $16$ possible ordered codes, which we label as follows:
	\begin{alignat*}{4}
		c_1&=(a_3,a_3,a_1,b_3),    &\quad c_2&=(a_2,b_2,b_1,b_2), &\quad c_3&=(a_2,b_2,a_2,a_1), &\quad c_4&=(b_3,b_3,a_3,b_1), \\
		c_5&=(a_3,b_2,a_3,b_2),    &\quad c_6&=(a_2,b_1,a_2,b_3), &\quad c_7&=(a_1,b_2,a_3,b_2), &\quad c_8&=(a_2,b_3,a_2,b_3), \\
		c_9&=(a_2,b_1,b_2,b_3), 	  &\quad c_{10}&=(a_3,b_2,b_2,b_3), &\quad c_{11}&=(a_2,b_3,a_3,a_2), &\quad c_{12}&=(a_1,b_2,a_3,a_2), \\
		c_{13}&=(b_1,b_1,a_2,b_2), &\quad c_{14}&=(a_3,b_3,a_3,a_2), &\quad c_{15}&=(a_3,b_3,b_2,b_3), &\quad c_{16}&= (a_1,a_1,a_2,b_2). 
	\end{alignat*}
	We construct a digraph $D$ (depicted in Figure~\ref{fig:digraph}) with the vertex set $c_1,\dots,c_{16}$ and an arc 
	between vertices $c_i$ and $c_j$ if the code $c_i$ can be connected to the code $c_j$.
	\begin{figure}[htp!]
		$$
			\includegraphics{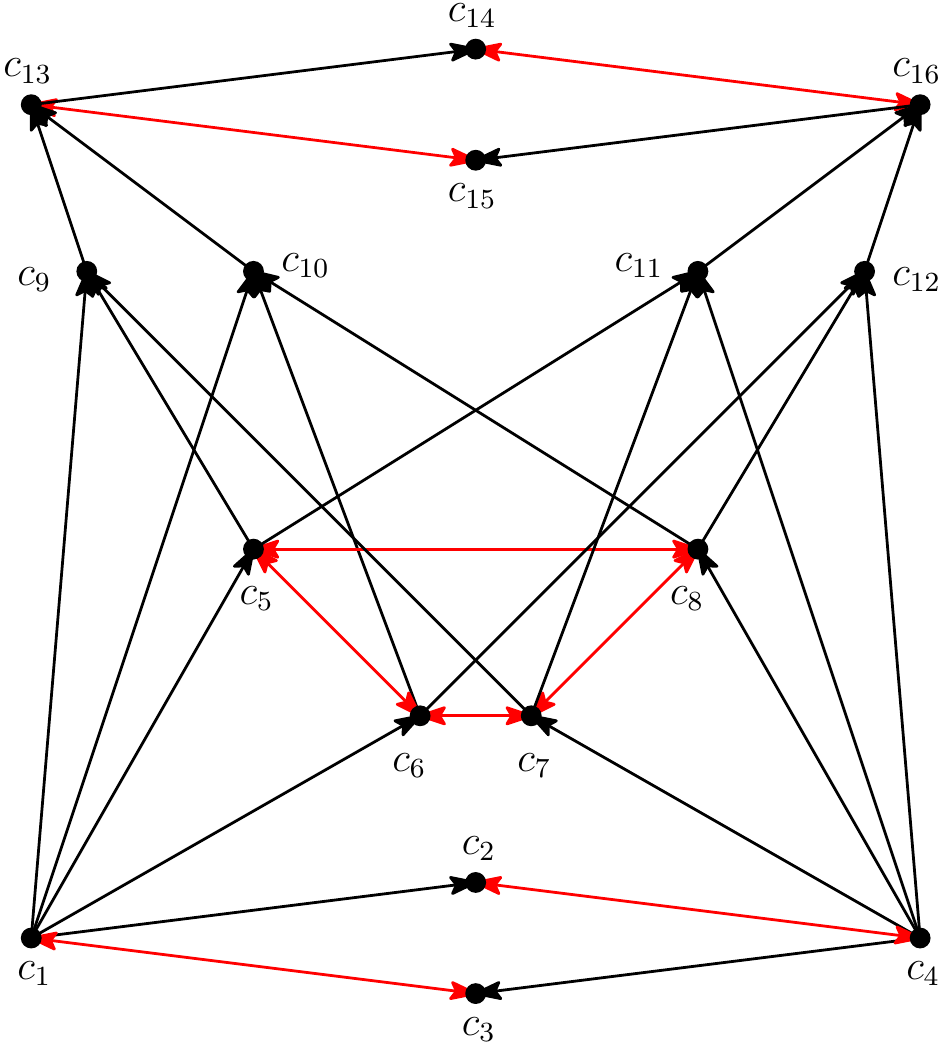}
		$$
		\caption{The digraph $D$. In red, we depict bidirectional arcs.}
		\label{fig:digraph}
	\end{figure}

	Note that $D$ consists of seven strong components (for strong connectivity of digraphs we refer to~\cite{Bang-JenGut09}, Section 1.5)
	with the vertex sets 
	$C_1= \{ c_1, c_2, c_3, c_4\}$, 
	$C_2= \{ c_5, c_6, c_7, c_8\}$, 
	$C_3= \{c_9\}$, 
	$C_4= \{c_{10}\}$, 
	$C_5= \{c_{11}\}$, 
	$C_6= \{c_{12}\}$, and 
	$C_7= \{c_{13}, c_{14}, c_{15}, c_{16}\}$. 
	Moreover, all strong components are bipartite, and therefore $D$ does not contain any oriented odd cycle. 
	This in turn means that there is no $2$-LIEC of $\mathrm{XI}_{2k+1}$ for any $k\geq 1$. 
	Note that, by \cref{thm:ring}, $\lir(\mathrm{XI}_{2k+1})=3$, since $\mathrm{XI}_n$ is a cycle permutation graph. 
\end{proof}

Theorem~\ref{thm:theoremXI} shows that there exists an infinite family of cubic graphs with girth 4 which do not admit a $2$-LIEC. 
So, one may ask whether having girth at least $5$ suffices for a cubic graph to admit a $2$-LIEC.
Using computer, we tested all cubic graphs with girth at least $4$ on at most 24 vertices,
and determined the number of graphs which do not admit a $2$-LIEC (see Table~\ref{tablenon2-liec}). 

\begin{table}[htp!] 
	\centering
	\begin{tabular}{|c|c|c|c|c|c|c|c|c|c|c|}
		\hline
		$g(G)\backslash n$ 	& 6	& 8	& 10	& 12 	& 14 	& 16 	& 18 	& 20 	& 22 	& 24 	\\
		\hline
		$\ge 4$		& 0 & 0 & 1 	& 2 	& 2 	& 0 	& 1 	& 0 	& 4 	& 0 	\\
		\hline
		$\ge 5$ 	& - & - & 0 	& 1 	& 2 	& 0 	& 0 	& 0 	& 2 	& 0 	\\
		\hline
	\end{tabular}
	\caption{The number of cubic graphs $G$ on small number of vertices and girth at least $4$/at least $5$ with $\lir(G) = 3$.}
	\label{tablenon2-liec}
\end{table}

We found five graphs with girth $5$, which do not admit a $2$-LIEC; 
two of them are the generalized Petersen graphs $GP(7,2)$ and $GP(11,2)$, depicted in Figure~\ref{fig:g5GP},
and the other three are depicted in Figure~\ref{fig:g5NGP}. 
Note that these five graphs are the only known not $2$-LIEC colorable cubic graphs of girth at least $5$.
\begin{figure}[htp!]
	$$
		\includegraphics{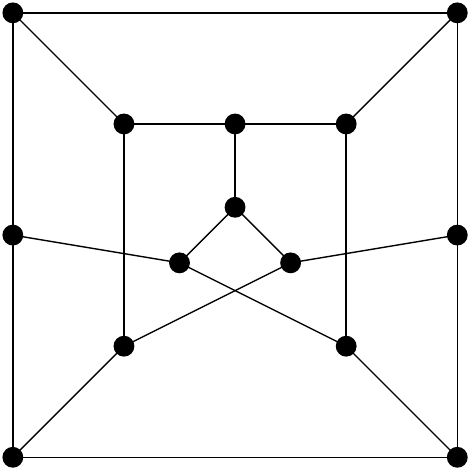} \quad\quad
		\includegraphics{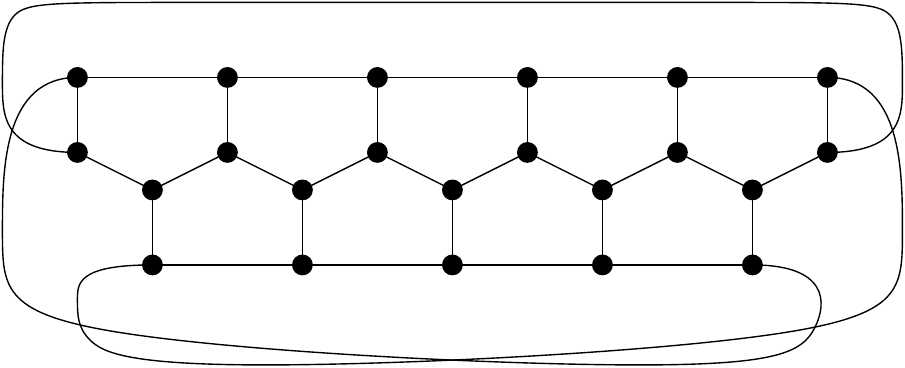}
	$$
	\caption{The generalized Petersen graphs $GP(7,2)$ (left) and $GP(11,2)$ (right) do not admit a $2$-LIEC.}
	\label{fig:g5GP}
\end{figure}

\begin{figure}[htp!]
	$$
		\includegraphics{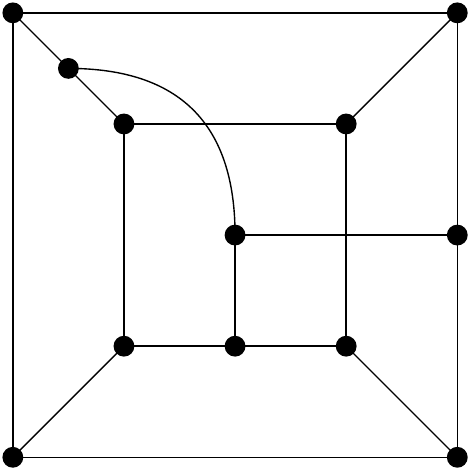} \hspace{2cm}
		\includegraphics{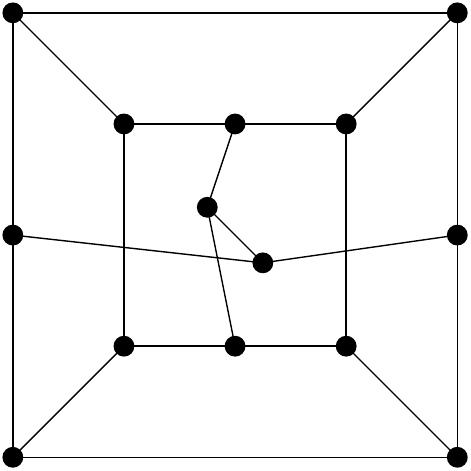}
	$$
	$$
		\includegraphics{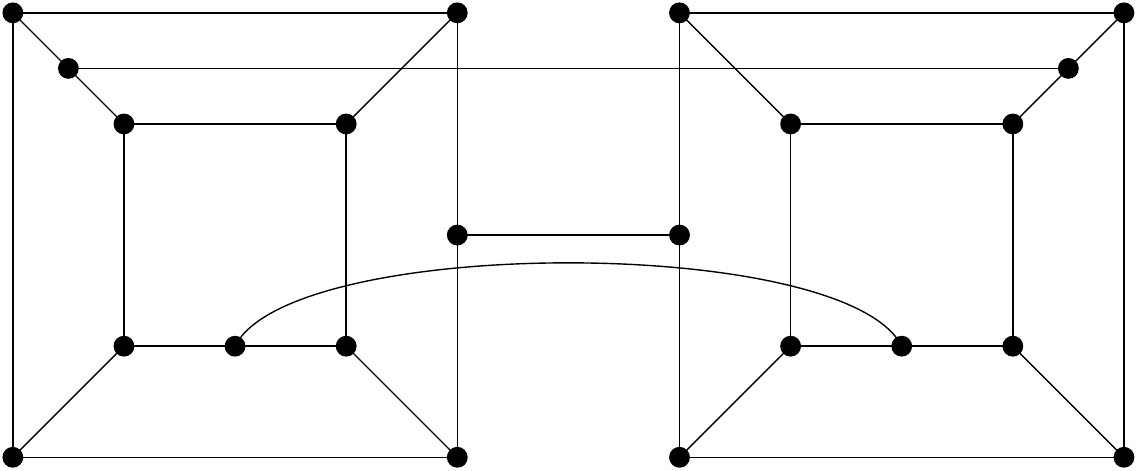}
	$$
	\caption{Three cubic graphs of girth $5$ on $12$, $14$, and $22$ vertices, which are not generalized Petersen graphs and do not admit a $2$-LIEC.}
	\label{fig:g5NGP}
\end{figure}
We computationally verified that these two are the only generalized Petersen graphs on at most $46$ vertices
that do not admit a $2$-LIEC and this encouraged us to propose the following.
\begin{conjecture}
	\label{conj:GenPet}
	For every generalized Petersen graph $G$ with girth at least $5$, with the exception of $GP(7,2)$ and $GP(11,2)$,
	we have that $\lir(G) = 2$.
\end{conjecture}

Based on computational evidence, we also believe the following.
\begin{conjecture}
	There exist $g_0$ and $n_0$ such that every connected cubic graph $G$ with girth $g(G) \geq g_0$ 
	and $|V(G)| \geq n_0$ admits a $2$-LIEC.
\end{conjecture}
As mentioned above, $g_0 \ge 5$ and if $g_0 = 5$ then $n_0 > 22$. We also note that the conjecture might be true 
for large enough odd girth, i.e., for graphs with short even cycles allowed.

\section{Conclusion}

In this paper we proved that decomposable subcubic graphs from some particular classes admit a $3$-LIEC
and many of them even a $2$-LIEC.
While there are still many decomposable subcubic graphs for which the bound of $3$ colors remains open,
we do believe that $3$ is the correct bound and, in light of Conjecture~\ref{conj:main2},
we propose a weaker version.
\begin{conjecture}
	\label{conj:subcub}
	For every decomposable graph $G$ with maximum degree $3$ it holds that $\lir(G) \le 3$.
\end{conjecture}

Regarding locally irregular edge-coloring, 
many questions remain open;
we conclude by recalling two of them regarding bipartite graphs.
\begin{question}[Baudon, Bensmail, and Sopena~\cite{BauBenSop15}, Question~5.1]
	Is the problem of deciding whether a bipartite graph admits a $2$-LIEC NP-complete?
\end{question}

\begin{question}[Bensmail, Merker, and Thomassen~\cite{BenMerTho16}, Question~5.1]
	Does there exist a bipartite graph $G$ with minimum degree $3$ and $\lir(G) > 2$?
\end{question}
Recall that there are bipartite graphs of minimum degree $2$ and arbitrarily large girth which do not admit a $2$-LIEC (see Figure~\ref{fig:4t+1}).

\paragraph{Acknowledgment.} 
M. Macekov\'{a}, S. Rindo\v{s}ov\'{a}, R. Sot\'{a}k, and K. Srokov\'{a} acknowledge the financial support from the projects APVV--19--0153 and VEGA 1/0574/21.
B.~Lu\v{z}ar and K. \v{S}torgel were partially supported by the Slovenian Research Agency Program P1--0383 and the projects J1--3002 and J1--4008.
K. \v{S}torgel also acknowledges support from the Young Researchers program.

\bibliographystyle{plain}
{
	\bibliography{References}
}

\end{document}